\documentclass[eqnumis, pdf]{hjms}
\usepackage{amsfonts}
\usepackage{amsmath}
\usepackage{footnote}
\usepackage{multicol}
\usepackage{booktabs}
\usepackage[flushleft]{threeparttable}
\usepackage[font=small,labelfont=bf]{caption}

\begin{document}
\hinfo{XX}{x}{XXXX}{1}{\lpage}{10.15672/hujms.xx}{Article Type}
%

\markboth{S.A. Aghili, H. Mirebrahimi, A. Babaee}{Targeted Complexity}

\title{On the Targeted Complexity of a Map}

\author{Seyed Abolfazl Aghili\coraut$^{1}$, Hanieh Mirebrahimi$^1$, Ameneh Babaee$^1$}

\address{$^1$ Department of Pure Mathematics, Ferdowsi University of Mashhad, Mashhad, Iran, P.O. Box 91775-1159.}
\emails{s.a.aghili@mail.um.ac.ir (S.A. Aghili), h\_mirebrahimi@um.ac.ir (H. Mirebrahimi), \\ ambabaee057@gmail.com (A. Babaee)}
\maketitle

\begin{abstract}
We study the topological complexity of work maps  with respect to some subspaces of the configuration space  and a workspace considered as the target set of the motion of robots. The motivation is to optimize and reduce the number of motion planners for work maps. In this regard, we  focus on the useful set of works. We  check some basic properties of the targeted complexity of maps, such as homotopical invariance, reduction, the product of maps, and so on. Then we compare these targeted complexities, and we find some  inequalities in reducing the number of motion planners.  We show that the relative topological complexity of pair of spaces defined by Short is a special case of the targeted complexity of work maps.
\end{abstract}

\hinfoo{DD.MM.YYYY}{DD.MM.YYYY} 


\section{Introduction and motivation}

In the modern world, robots play important roles in the development of industry and technology. Using robots as human resources is inevitably important in several ways: reducing human errors, increasing the  performance accuracy, reducing the likelihood of human-health hazards, and so on. Indeed, applying robots, despite the benefits, is not easy, and many issues need to be resolved.

One of the most important steps for using robots is  solving the motion planning problem. The motion planning problem is a construction of a program for moving a robot  from one point to another, formulated as follows: Let the robot can move at all points  of the connected space $X$.  
Each motion planning corresponds to a program that defines a connecting path for any pair of points.
In most sources  that studied topological complexity, configuration spaces are considered to be path-connected, and then for every pair of points, there exists some path connecting them. In robotic motions, it is important that the correspondence establishing paths is continuous. It means that for any two close pairs of points $(A, B)$ and $(A', B')$, their corresponded paths $s(A, B)$ and $s(A', B')$ are also close. The closeness of paths is defined by the topology of path space $PX$. The number of discontinuities of the motion planning is called the topological complexity of space $X$, denoted by $TC(X)$. Farber \cite{bib11} supposed any robot as a point in the motion space. Then other authors \cite{bib2, bib9, bib7, bib1} modified that assumption and considered two distinct spaces:  the space of all states of the robot, called the configuration space, and the space of all robot tasks, called the workspace of the robot; see Figure \ref{fig00}. Then they generalized the definition of topological complexity for any given work map. Pave\v{s}ic \cite{bib1, bib10}, Rami and Derfoufi \cite{bib5}, Murillo and Wu \cite{bib2}, and so on defined and studied the topological complexities of maps, each with its own advantages, applications, and lacks.
The topological complexity of maps was generalized to the higher topological complexity of maps denoted by $TC_n (f)$; for more information, see \cite{Is, bib6}.

In this paper, we study and modify the topological complexity of maps defined by Pave\v{s}ic \cite{bib1}, Scott \cite{bib7}, and Is et al. \cite{Is}, with respect to some subsets of the configuration and workspaces. These sets are assumed to be the targets of the motion of robots; that is, the endpoint of the motion necessarily belongs to the target set. In fact, if the  target set equals the whole space, then the targeted complexity is the same complexity as was defined originally. 
 Scott \cite[Definition 3.1]{bib7} defined a homotopy invariant topological complexity  $TC^S$ and  proved that $TC^{MW}$, the topological complexity defined by Murillo and Wu \cite{bib2}, is equivalent to the notion $TC^S$. Also, he  \cite[Examples 3.3 and 3.4]{bib7} showed with examples that $TC^{RD}(f)$, the topological complexity defined by Rami and Derfoufi \cite{bib5}, and $TC^{P}(f)$ are not homotopy invariant. 


In Section 2, we recall some preliminaries, such as  definitions of topological complexities of maps and some of their basic properties. We need the definition of these notions and their properties for the rest of the paper.
Then in Section 3, we consider the complexity of Pave\v{s}ic with the target that depends on a subspace of the given configuration space. We call the motion and complexity  with some fixed task sets the targeted motion and targeted complexity. A special case of targeted motion is the endpoint movement whenever the target set is a singleton; see \cite{endpoint}.
Using this complexity, one can choose subspaces of the configuration space and workspace to focus on some special works done by a robot. Then only the useful set of works needs to have motion planners, and we can program the robot only to do the required work. 
Then we investigate some basic properties such as the targeted complexity of the composition and product of maps.

In Section 4, we compare the targeted complexity of maps with other  complexities and then  find several inequalities under some conditions. Also, we see that for the identity map, the targeted complexity coincides with the relative topological complexity of the pair of spaces defined by Short \cite{bib4}. Recall that for the identity map, the robot is considered a point in the configuration space, and then the configuration space and the workspace coincide.

In Section 5, we study the targeted complexity based on the Scott's complexity of maps. Then we check some properties of targeted complexity from this point of view.
In Section 6, we reformulate a kind of $n$-dimensional complexity and also a kind of $n$-dimensional  targeted complexity. Then we investigate  their basic properties, such as composition, product, and relation to the complexity of topological spaces. 

\section*{Motion planning}

It is well known that the topological complexity of spaces is a special case of the complexity of maps; that is, the complexity of identity map. In robotic motions, the identity map as the work map means that we consider the robot moves as a point in the configuration space. Recall that in this case, the configuration space and the workspace are assumed to be the same. For the work maps, the section maps prepare motion planners for the motion of robot as a map from the state of robot to the corresponded task. Let $f:X \to Y$ be a continuous map. The complexities of maps studied in this work provide motion planners as follows.
\begin{itemize}
\item
For the Pave\v{s}ic's complexity, for each pair of points $(x, y) \in X \times Y$, the motion planner  offers a path from $x$ to $x'$, where $f(x') = y$. In other words, we intend to find an algorithm, namely, the section $s: U \subseteq X \times Y \to PX$, for achieving from a given state (point $x$) to a required task (point $y$). By this method, we move in the configuration space to attain some states whose corresponded task we want to reach. In the targeted case, we choose some tasks and define motion planners just on this set of tasks, which can be restricted to the useful or special tasks we require. Let $B$ be a set whose image $f(B)$ is the needed set of tasks called target set. Then the section $s$ is defined from  the open set $U \subseteq X \times f(B)$ to the path space of $X$; that is, the endpoint of the path necessarily  belongs to the set $B$.

\item
For the map $f: X \to Y$, Scott considered the motion planner $s$ from the set $Z \subseteq X \times X$ to the path space $PY$ such that it corresponds to each pair of points $(x_0, x_1)$ a path $\alpha$ in $Y$ connecting $f(x_0)$ to $f(x_1)$.
That is, the motion planner is an algorithm to find a path to move from the task $f(x_0)$ to the task $f(x_1)$.
For the targeted case of Scott's complexity, the endpoint belongs to the target set $f(B)$. In fact, we find an algorithm to move from the task $f(x_0)$ to the task $f(b)$ ending in the target set $f(B)$. The Scott's complexity has a difference with the Pave\v{s}ic's complexity. The motion planners to obtain Scott's complexity move in the workspace $Y$, but motion planners for Pave\v{s}ic's complexity move in the configuration space $X$. Hence, to find the most efficient algorithm for motion planners, we first verify whether it is easier to move in the configuration space  or the workspace, and then choose one of the Pave\v{s}ic's complexity or Scott's complexity to use.

\item
In general, the $n$-dimensional complexities count the motion planners with paths consisting of $n$ pieces. One of the applications of $n$-dimensional complexity and breaking paths to subpaths is to recognize and eliminate repeated movements; see \cite{Bas, Bas2}. For example, Is et al. considered paths and pieces in the workspace. In Section 5, we reformulate this $n$-dimensional complexity such that paths and pieces belong to the configuration space. Then  we choose a target set for the movement and consider the motion planners such that subpaths end in the target set contained in the configuration space. Finally, we can choose the easier one, either the configuration space or the workspace, to find the best motion planners.

\end{itemize}

\begin{figure}[!ht]
\centering
\includegraphics[scale=.25]{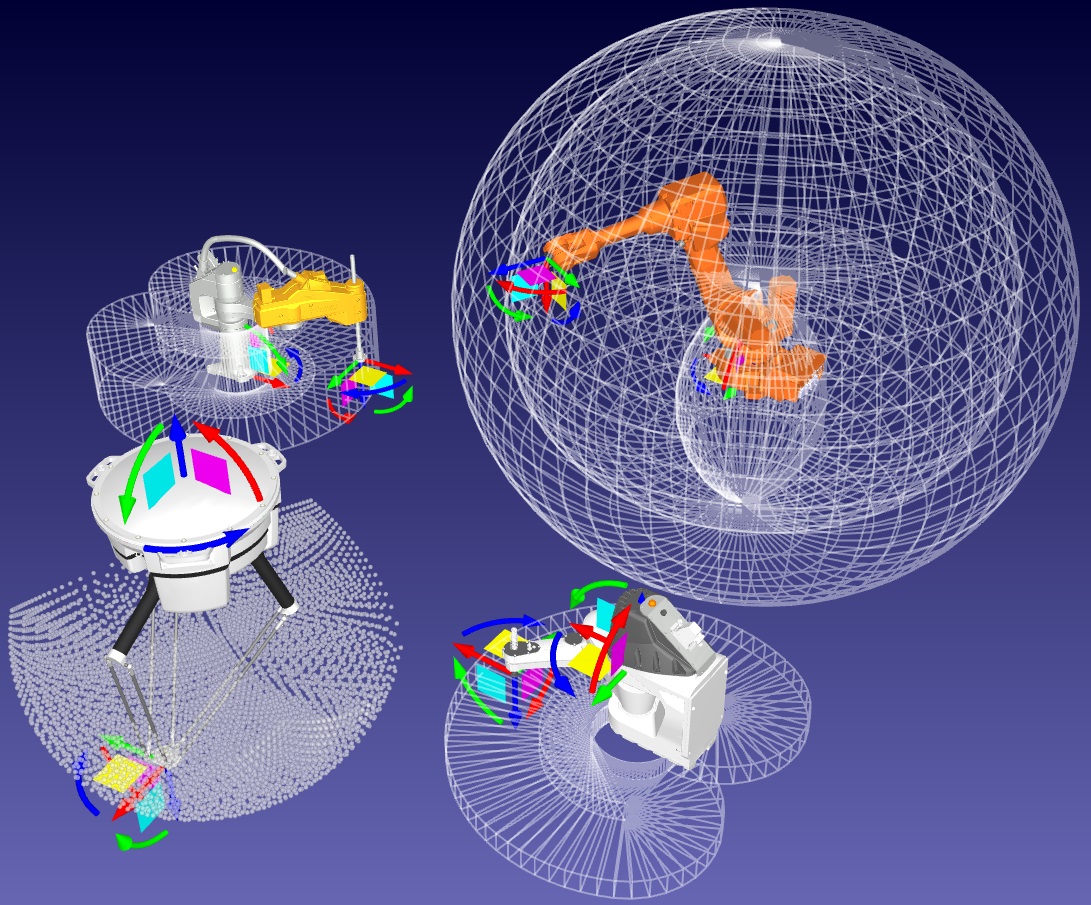}
\caption{Some examples of configuration space and workspace.}\label{fig00}
\end{figure}

\section{Preliminaries}
Let  $X$ be a path-connected topological space,  let $PX$ denote the space of all paths in $X$ equipped with the compact-open topology, and let $\pi : PX \to X \times X$ denote the natural path fibration mapping each path to its endpoints.   Farber \cite{bib11} defined the \textit{topological
complexity} of space $X$, denoted by $TC(X)$, as the minimum number of motion plannings. By a \textit{motion planning} we mean a continuous map $s_i : U_i \to PX$ with $\pi s_i = id$ over $U_i$, where $U_i \subseteq X \times X$.
If there is  no such integer $k$, then $TC(X)$ is considered to be the infinity, $TC(X) := \infty$. The map $s_i$ is called a \textit{section} for the fibration $\pi: PX \to X$. For any continuous surjection $p: E \to B $, by a \textit{total section} we mean a continuous right inverse of $p$, namely, a map $s: B \to E$ such that $p s = id_B$. Moreover, for any subspace $A\subseteq B$, a
\textit{section} of $p$ over $A$ is a section of the restriction map $p: p^{-1}(A) \to A$.

The notion of topological complexity of a given space can be expressed in connection with the sectional number of a special map. We recall the definition of sectional number in the following statement as defined in \cite[p. 3]{bib1}.

\begin{definition}
Let $p: E \to B $ be a continuous map. We define the \textit{sectional number} of $p$,  $sec(p)$, to be the minimal integer $n$ for which there exists an increasing sequence of open subsets
$\emptyset = U_0 \subset U_1 \subset U_2 \subset \cdots \subset U_k = B$
such that each difference $ U_i - U_{i-1}, i = 1, \ldots, n$, admits a continuous partial section
of $p$. If there is no such integer $n$, then  $sec(p)$ is considered to be $\infty$.
The topological complexity of a path connected space $X$ is
$TC(X) := sec(\pi)$, where $\pi : PX \to X \times X$ is the natural path fibration.
 \end{definition}

Robots have different structures and use; for instance, when the mechanical system is a robotic arm, a system consisting of some bars and flexible joints (Figure \ref{fi1}), the $TC(X)$ input on the navigation problem is not quite satisfactory.

\begin{figure}[!ht]
\centering
\includegraphics[scale=.8]{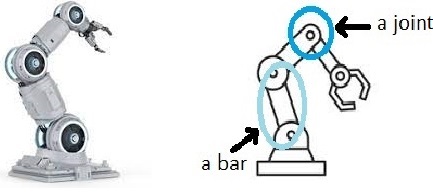}
\caption{Robotic arm.}\label{fi1}
\end{figure}

The configuration space of the mechanical system equals the set of all possible states of the robot, and the workspace equals the set of all states of doing some task. For instance, the workspace of a robotic arm equals the set of all points that may be  accessed by the end effector.
 In the case of a robotic arm, depending on the motion of the joints, the workspace equals the product of some spheres, cylinders, and so on; see Figure \ref{fi2}.

 \begin{figure}[!ht]
\centering
\includegraphics[scale=.5]{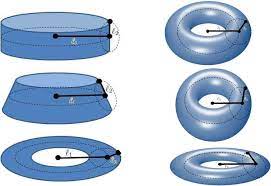}
\caption{Workspace of a robotic arm with two bars and two joints is the torus \cite{bib2}.}\label{fi2}
 \end{figure}
Farber \cite{bib11} considered robots as some points in the configuration space, and then the configuration space and the workspace coincide. Indeed Murillo and Wu \cite{bib2} separated these two concepts and generalized the usual topological complexity for any work map $ f: X \to Y$ from the configuration space $X$ to the workspace $Y$. Since for the usual topological complexity, the configuration space and the workspace are assumed to be the same, the work map is the identity; that is, $TC (X) = TC^{MW} (id_X)$ for any workspace $X$. In that case, Dranishnikov  in his lectures at the workshop on Applied Algebraic Topology in Castro Urdiales (Spain, 2014)  suggested to studying the topological complexity of maps $TC(f)$ and defined it as the minimal number $k$ that $X\times Y$ can be covered by open sets $U_0, \ldots,U_k$ such that  over each $U_i$, there is a section for the map $q: X^I \to X \times Y$ defined as $q(\phi)=(\phi(0),f\phi(1))$ (see \cite{bib10}). The drawback of
this definition is that it can be applied only to maps $f:X \to Y$  having the path lifting property. All path fibrations and the left divisor of path fibrations have this property, but it is unclear whether forward kinematic maps for general robotic arms are left divisors of fibrations.

Pave\v{s}ic \cite{bib1,bib10} modified the following definition to cover general maps.
\begin{definition}[\cite{bib1}]
Let $f: X \to Y $ be a continuous map between path-connected spaces and let
$\pi_f : X^I \to X \times Y$ be defined as $\pi_f(\alpha) = (\alpha(0), f \alpha(1)) $. Then the
topological complexity of the map $f$ is defined as
$TC^{P}(f) := sec(\pi_f )$. Clearly $TC^{P}(id_X) = TC(X)$.
\end{definition}
Indeed, still, his definition lacks one important feature: It is not a homotopy invariant.
 Another version of the topological complexity of a map was given by Rami and Derfoufi \cite{bib5}, which is denoted by
$TC^{RD}(f)$. This version is also not a homotopy invariant.
Recently when this work was in progress,  Murillo and Wu \cite{bib2} defined their own notion of $TC(f)$, denoted by $TC^{MW}(f)$, which is a homotopy invariant.
They used the concept $Secat_f(p)$ to define this notion from $TC(f)$ as follows.
\begin{definition}[\cite{bib2}]
Let $p:E \to B$ and $f:B \to X$ be two continuous maps.
 An open set $U\subseteq B$ is $f$-categorical if there is a section $s: U\to E$ such that $f p s \simeq f\vert_U$. The $f$-sectional category of $p$, denoted by $Secat_f (p)$, is the least $n \leq \infty $ for which $B$ admits a covering of $f$-categorical open sets.
\end{definition}
Let $f:X\rightarrow Y$ be a continuous map. Then $TC^{MW}(f)=Secat_{f\times f}(\pi)$,
where $ \pi: X^I \to X \times X $ is the path fibration defined by  $\pi(\alpha)=(\alpha(0),\alpha(1))$; see \cite{bib2}.

Murillo and Wu \cite{bib2} gave a cohomological lower bound for this invariant. Notably, with the perspective Murillo and Wu took, they were unable to prove the inequality $TC(f) \leq TC(Y)$.
 Scott \cite[Definition 3.1]{bib7} defined a homotopy invariant notion for the topological complexity of
a map $f:X \to Y$, denoted by $TC^S(f)$, and he proved that $TC^{MW}(f)$ is equivalent to the notion $TC^S(f)$.  Moreover, he proved the basic inequality $TC^S(f)\leq \min \lbrace TC(X), TC(Y) \rbrace$.
The topological complexity of a map from point of view of Scott is defined as follows.
\begin{definition} [\cite{bib7}]
Let $f: X \to Y$ be a map.\\
1) An $f$-motion planner on a subset $Z \subseteq X \times X$ is a map $f_Z : Z \to Y^I$ such that
$f_Z(x_0, x_1)(0)= f(x_0)$ and $f_Z(x_0, x_1)(1)= f(x_1)$.\\
2) An $f$-motion planning algorithm is a cover of $X \times X$ by sets $Z_0, \ldots,Z_k$ such that, for each
$Z_i$, there is an $f$-motion planner $f_i: Z_i \to Y^ I$.\\
3) The (pullback) topological complexity of $f$, denoted by $TC^S(f)$, is the least integer $k$  that $X\times X$
can be covered by $k+1$ open subsets $U_0, \ldots,U_k$ on which there are $f$-motion planners. If
no such $k$ exists, then  $TC^S(f)= \infty$.
\end{definition}

Scott defined the mixed topological complexity of $f$, denoted by $TC^{\frac{1}{2}}(f)$, and  proved the inequalities
\begin{equation}\label{eq}TC^S(f)\leq TC^{\frac{1}{2}}(f)\leq TC^{P}(f).\end{equation}

\begin{definition}[\cite{bib7}]
Let $f : X \to Y$ be a map.
\begin{enumerate}
\item
 A mixed $f$-motion planner on a subset $Z \subseteq  Y \times X$ is a map $f_Z : Z \to Y^I$ such that $f_Z(y, x)(0) = y$ and $f_Z(x_0, x_1)(1)= f(x_1)$.

\item
A mixed $f$-motion planning algorithm is a cover of $Y \times X$ by sets $Z_0, \ldots, Z_k$ such that on
each $Z_i$, there is a mixed $f$-motion planner $f_i : Z_i \to Y^I$.
\item
 The mixed topological complexity of $f$, denoted by $TC^{\frac{1}{2}} (f)$, is the least $k$ such that $Y \times X$ can
be covered by $k + 1$ open subsets $U_0, \ldots, U_k$ on which there are mixed $f$-motion planners.
If no such $k$ exists, we define $TC^{\frac{1}{2}} (f) = \infty$.
\end{enumerate}
\end{definition}

Scott used the homotopy invariance of $TC^S(f)$ to define the topological complexity of group homomorphisms.

 Short \cite{bib4} defined the relative topological complexity of pair of spaces $(X, Y)$, where $Y\subseteq X$,  as a new variant of relative topological complexity.

\begin{definition}[\cite{bib4}]
Let $f: E\to B$ be a fibration. The Schwarz genus of $f$, denoted by $genus(f)$, is the smallest integer
$k$ such that there exists $\{ U_i \}_{i=1}^k$, an open cover of $B$, along with sections $s_i: U_i\to E$ of $f$.
\end{definition}

The relative topological complexity of a pair of spaces $(X, Y)$, denoted by $TC(X, Y)$, was defined as the Schwarz genus of a natural path fibration map. The relative topological complexity of a pair of spaces is a lower bound for the classic topological complexity defined by Farber \cite{bib11}. In fact, it is modified to count the number of rules presenting paths whose endpoints are limited.

\begin{definition}[\cite{bib4}]
Let $B\subseteq X$ and let $P_{X\times B}= {\lbrace \gamma \in PX \vert \gamma(0) \in X,\ \gamma(1) \in B \rbrace}$. There is a natural fibration $\pi_B: P_{X\times B} \to X \times B$ with $\pi(\gamma)= (\gamma(0), \gamma(1))$. Then, the relative topological complexity of the pair $(X,B)$ is the Schwarz genus of $\pi_B$.
 That is, $TC(X,B)= genus(\pi_B)= Secat(\pi_B)$.
\end{definition}

The concept of homotopic distance between two maps was introduced in \cite{Mac} as follows.
\begin{definition}[\cite{Mac}]
Let $f, g : X \to Y$ be two continuous maps. The homotopic distance $D(f, g)$ between $f$ and $g$ is the
least integer $n \geq 0$ such that there exists an open covering ${U_0, \ldots , U_n}$
of $X$ with the property that $f|_{U_j} \simeq g|_{U_j}$, for all $j = 0, \ldots , n$. If there is
no such covering, we define $D(f, g) = \infty$.
\end{definition}
Note that the  homotopic distance depends only on the homotopy class. That is,  if $f \simeq f^ \prime$ and $g \simeq g^\prime$, then $D(f, g) = D(f^{\prime}, g^{\prime})$. Moreover\\
1)  $D(f, g) = D(g, f)$.\\
2) $D(f, g) = 0$ if and only if the maps $f$ and $g$ are homotopic.\\
Also, for the composition of maps, we have the following inequality.
\begin{proposition}[\cite{Mac}]\label{e}
Consider maps $f,g:X\rightarrow Y$ and $h:Y\rightarrow Z$. Then
$D(hof,hog)\leq D(f,g)$.
\end{proposition}

 In the next section, we study targeted complexity for work maps as a generalization of Short's topological complexity of pair of spaces.

\section{Targeted complexity of maps}
In this section, using the complexity of maps defined by Pave\v{s}ic, we recall the targeted complexity of maps with respect to some subspaces of the configuration space and workspace. Using this topological complexity, we omit the subset of workspace that is not needed. In fact, we consider a subspace of the configuration space and its image by the work map and count only the number of motion planners whose paths end in the given subspace.

Let $f:X \to Y$ be a map and let $B \subseteq X$.
Recall that the map $\pi_f : PX \to X \times Y$ was defined in \cite{bib1, bib10} by   $\pi_f(\gamma)= (\gamma(0), f(\gamma(1)))$.

\begin{definition}\label{d2}
Let $\pi_{f(B)}: P_{X\times B} \to X\times f(B)$ be defined by $\pi_{f(B)}(\gamma)= (\gamma(0), f(\gamma(1)))$. Then
 the targeted complexity of map, denoted by $TC^P(f,B)$, is
the least integer $n \geq 0$ such that there exists an open covering ${U_0, \ldots, U_n}$
of $X\times f(B)$, where $U_i$ admits a partial section of $\pi_{f(B)}$ for $i = 0, \ldots, n$. If there is no such covering, then we define $TC^P(f,B)= \infty$.
\end{definition}

Note that in the above definition, we have  $TC^P(f,B)= Secat(\pi_{f(B)})$. Moreover, if $f$ is a fibration map and  $B = X$, then $TC^P(f, B) = TC^P (f)$. Also, if $f := id_X$, then $TC^P(f, B) = TC(X, B)$; that is, the relative topological complexity of the pair $TC(X, B)$, defined by Short,  is a special case of Definition \ref{d2}.

In Definition \ref{d2}, the targeted complexity of a map is defined as the sectional category number of  a map, restricted to some subspace. The original map was previously  assumed by Pave\v{s}ic.
He showed that for homotopic fibration maps, the complexities are equal. For the targeted complexities, it holds if the target subspaces are equal as follows.

\begin{proposition}\label{pr4.13n}
If $f,g:X \to Y$ are homotopic fibrations and $f(B) = g(B)$, then $TC^P(f,B)=TC^P(g,B)$.
\end{proposition}

%
%

In the next section, we show that $TC^{P} (f, B) = 0$ if and only if $f$ is nullhomotopic; see Corollary \ref{co4.6nn}. 
In the following proposition, we show that the targeted complexity increases by  locally injective left compositions.

\begin{proposition}\label{p1}
Let $f:X \to Y$ and $g:Y\to Z$ be two maps, let $B\subseteq X$, and let $g$ be injective on $B$. Then
$ TC^P(f,B)\leq TC^P(gf,B) $.
\end{proposition}

\begin{proof}
Let $TC^P(gf,B)= n$. Then for some open subsets $U_0, U_1, \ldots, U_n $ of $X \times gf(B)$, there are $s_i: U_i \to P_{X\times B}$ such that $\pi_{gf(B)} s_i= id_{U_i} $. We set $V_i= (1\times g)^{-1} (U_i) \subseteq X\times f(B)$ and define $s_i^{\prime}: V_i \to P_{X\times B}$ by  $s_i^{\prime}(x,y)= s_i(1\times g)(x,y)$. Hence we  show that $\pi_{f(B)} s_i^{\prime}= id_{V_i}$. For this purpose, for all $(x,y)\in V_i$ with $x\in X$ and $y=f(b)\in f(B)$, we have $\pi_{f(B)} s_i^{\prime}(x,y)= \pi_{f(B)} s_i(1\times g)(x,y)= \pi_{f(B)} s_i(x,gf(b))= (x,f(b)=y)$. Therefore $ TC^P(f,B)\leq n$.
\end{proof}

Farber as the first author who studied topological complexity, investigated the topological complexity of product of spaces and presented the following inequality:
\[
TC(X \times Y) \le TC(X) + TC(Y).
\]
Now we intend to find a similar version. Let $f$ be an open map. Then the following proposition holds.

\begin{proposition}
Let $f:X\to Y$ and $g:Z\to W$ be two fibration maps, where $Y$ and $W$ are normal spaces, let $B\subseteq X,$, and let $ D\subseteq Z$. Then $TC^P(f \times g, B\times D) \leq TC^P(f,B) + TC^{P}(g,D)$.
\end{proposition}
\begin{proof}
By Definition \ref{d2},  $TC^P(f,B)= Secat(\pi_{f(B)})$ and $TC^P(g,D)=Secat(\pi_{g(D)})$, where $\pi_{f(B)}: P_{X\times B}\to X\times f(B)$ and $\pi_{g(D)}: P_{Z\times D}\to Z\times g(D)$ are natural path fibrations. On the other hand, $TC^P(f\times g,B\times D)= Secat(\pi_{f\times g(B\times D)})$, where $\pi_{f\times g(B\times D)}: P_{(X\times Z)\times (B\times D)}\to (X\times Z)\times (f(B)\times g(D))$ is a path fibration. One can check that two maps $\pi_{(f \times g) (B \times D)}$ and $\pi_{f(B)} \times \pi_{g(D)}$ can be considered to be equal.
Therefore, $Secat(\pi_{f\times g(B\times D)})= Secat(\pi_{f(B)}\times \pi_{g(D)})$. Moreover, by \cite[Proposition 2.4]{bib7},  the inequality $secat (p\times p^{\prime})\leq secat p + secat p^{\prime} $ holds for each pair of fibration maps.
\end{proof}

\section{Comparison with other topological complexities}

In this section, we compare the targeted complexity with other topological complexities, for instance with the relative topological complexity of pair of spaces defined by Short and with  Scott's topological complexity of maps. As a natural fact, the targeted complexity is less than the Pave\v{s}ic's  complexity of a map as follows.
Note that if $f$ is a fibration, then $TC^{P}(f)= Secat(\pi_f)$ (see \cite{bib1}).

\begin{proposition}\label{pr4.7n}
If $f:X\to Y$ is a fibration, then
$TC^P(f,B)\leq TC^{P}(f)$.
\end{proposition}

\begin{proof}
Similar to the proof of \cite[Corollary 2.6]{bib4}, $\pi_{f(B)}$ is the pullback of the map $\pi_f$ induced by the inclusion map $X \times f(B) \hookrightarrow X \times Y$. Then by \cite[Proposition 2.5]{bib4}, $Secat (\pi_{f(B)}) \le Secat (\pi_f)$, and the inequality holds.

\end{proof}

\begin{example}
\begin{itemize}
\item
If $B=\emptyset $, then $TC^P(f,B)= \infty$ for $f : X \to Y$.
\item
If $B$ is a singleton set, then $TC^{P}(f, B) = cat (f)$ for any map $f:x \to Y$.
Recall that
the category of a map $f$, $cat(f)$, is the least $n \le \infty$
such that the domain can be covered by $n + 1$ open sets on each of
which the restriction of $f$ is homotopically trivial.
\item
If $f:X \to Y$ is a constant map, then for any $B \subseteq X$, $TC^P(f, B) = 0$, because $TC^p(f, B) \le TC^P(f) = 0$ by \cite{bib1}.
\item
Consider the space $X = S^2$ and $B$ as the upper half of the sphere. Let $f:X \to X$ be the map contracting $B$ to the north pole of the sphere  (Figure \ref{fi5.1nn}). Then $TC^{P} (f, B) = cat (f) = cat (S^2) = 1$.

\end{itemize}

\begin{figure}[!ht]
\centering
\includegraphics[scale=.6]{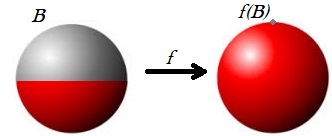}
\caption{Map contracts the upper half to the north pole of the sphere.}\label{fi5.1nn}
\end{figure}

\end{example}

In the following proposition, we compare  the relative topological complexity  of a pair of spaces $(X,B)$ and the targeted complexity of a map for a given injective map.

\begin{proposition}\label{pr4.4n}
Let $f: X\to Y$ be an injective map and let $B\subseteq X$. Then
$TC(X,B) \leq TC^P(f,B)$ and the equality holds if $f$ is the identity map.
\end{proposition}
\begin{proof}
Let $ TC^P(f,B)= n$. Then for some open subsets $U_0, U_1, \ldots, U_n $ of $X\times f(B)$, there are $s_i: U_i \to P_{X\times B}$ such that $\pi_{f(B)} s_i=id_{U_i} $.
Set $V_i= (1\times f)^{-1}(U_i)$ contained in $X\times B$, since $f$ is injective. Moreover, define $s_i ^{\prime}: V_i \to P_{X\times B}$ by  $s_i^{\prime}(x_1, x_2)= s_i (1\times f) (x_1, x_2)$. We  show that $\pi_B s_i^{\prime}= id_{V_i}$, where $\pi_B: P_{X\times B} \to X\times B$ maps each path to its start and endpoints. Let $(x_1,x_2)\in V_i$.  Then $\pi_B s_i^{\prime}(x_1,x_2)= \pi_B s_i (1\times f)(x_1,x_2)= \pi_B s_i (x_1, f(x_2))$. Note that $s_i$ maps each $(x, y)$ to a path in $X$ whose endpoint, namely $x_3$, belongs to $B$. Moreover, since $\pi_{f(B)} s_i=id_{U_i} $,  so $f(x_3) = y$. Thus $s_i (x_1, f(x_2))$ is a path in $X$ starting at $x_1$ and ending at $x_3$, where $f(x_3) = f(x_2)$. Since $f$ is an injection, $x_2 =  x_3$. Also, since $\pi_B$ maps each path to its endpoint, $\pi_B s_i (x_1, f(x_2)) = (x_1, x_3) = (x_1, x_2)$. Then  $\pi_B s_i^{\prime} = id_{V_i}$. Therefore $TC(X,B) \leq TC^P(f,B)$.

Now let $f = id_X$. Then in Definition \ref{d2}, we have $f(B) = B$ and $\pi_{f(B)}= \pi_B$. Hence $TC^{P} (f, B) = Secat (\pi_{f(B)}) = Secat (\pi_B) = TC (X, B)$.
\end{proof}

Furthermore, in the following proposition, we present an inequality for fibration maps.


\begin{proposition}
Let $f: X\to Y$ be a fibration map, let $B\subseteq X$, and let $f(B)=C\subseteq Y$. Then $TC^P(f,B)\leq TC(Y,C)$.
\end{proposition}

\begin{proof}
Since $\pi_{f(B)}$ is the pullback of the fibration $\pi:P_{Y \times C} \to Y \times C$ induced by $f: (X, B) \to (Y, C)$,  the inequality $Secat(\pi_{f(B)}) \le Secat (\pi)$ holds  by \cite[Proposition 2.3]{bib1}. Then  $TC^{P}(f, B) \le  TC(Y, C)$.
\end{proof}

In the following proposition, we compare $TC^S(f)$ with the targeted complexity of a map.
\begin{proposition}\label{pr4.9n}
Let $f: X\to Y$ be a map and let $B\subseteq X$. Then $TC^S(f)\leq TC^P(f,B)$.
\end{proposition}

\begin{proof}
Let $TC^P(f,B)=n$. Then for some open subsets $U_0, U_1, \ldots, U_n $ of $X\times f(B)$, there are $s_i: U_i \to P_{X\times B}$ such that $\pi_{f(B)} s_i=id_{U_i} $. We set $V_i= (1\times f)^{-1}(U_i)\subseteq X\times X$ and define $f_{V_i}: V_i \to Y^I$ by $f_{V_i}(x_1,x_2)= f_\ast s_i (1\times f)(x_1,x_2)$ in which $f_\ast: X^I \to Y^I$ is defined by $f_\ast(\gamma)= f\gamma$. Therefore $TC^S(f)\leq n$.
\end{proof}

Many kinds of topological complexities are trivial for the null objects. In the following result, we find a two-sided statement for the targeted complexity of maps.

\begin{corollary}\label{co4.6nn}
Let $f:X\to Y$ be any fibration map and let $B\subseteq X$. Then $TC^P(f,B)=0$ if and only if $f$ is nullhomotopic.
\end{corollary}
\begin{proof}
If $f$ is a nulhomotopic map, then   Proposition \ref{pr4.13n} implies $TC^P(f, B)= 0$. For the converse, using Proposition \ref{pr4.9n}, we have  $TC^S(f)=0$. Therefore it follows from \cite[Corollary 3.9]{bib7} that $f$ is nullhomotopic.
\end{proof}

If $f$ is a fibration, then it follows from Propositions \ref{pr4.7n} and \ref{pr4.9n} that $TC^S(f) \le TC^P(f, B) \le TC^{P} (f)$.
Also, if $f$ admits a homotopy section, then by using \cite[Corollary 3.8]{bib1}, we have $TC^{P} (f) = TC(Y)$. Moreover, if $f$ admits a homotopy section, then $TC^{MW} (f) = TC (Y)$ by \cite[Proposition 2.4]{bib4}. In other words, the following corollary holds.

\begin{corollary}
Let $f:X\to Y$ be a fibration map and admit a homotopy section and let $B\subseteq X$. Then
\[
TC^{P}(f,B) = TC^S(f)= TC^{MW}(f)= TC^{\frac{1}{2}}(f)= TC^{P}(f)=  TC(Y).
\]
\end{corollary}

\section{Scott's complexity with target}

In this section, we study Scott's complexity of maps with some targets. In this regard, we choose the motions whose endpoints belong to the target subset.
Let $f : (X,B) \to (Y,C)$ be a map of pairs of spaces. Then one can consider Scott's complexity of a map together with a subset $B$ as the target, denoted by $TC^S(f,B)$.

\begin{definition}
Let $f: (X, B) \to  (Y, f(B))$ be a map of pairs. Then $TC^S(f, B)$ is the least integer $n\geq 0$ satisfying that $X\times B$ is covered by open subsets $U_0, U_1, \ldots, U_n$ and there is $\sigma_i: U_i\to P_{Y\times f(B)}$, such that $\sigma_i(x,b)(0)= f(x)$ and $\sigma_i(x,b)(1)= f(b)$. If there is no such covering, then we define $TC^S(f, B)= \infty$.
\end{definition}

One of the motivations for defining the  targeted complexity is to optimize complexity of maps. Then at first  it is natural that the following inequality holds.

\begin{proposition}
Let $f:X\to Y$ be a map. Then $TC^S(f, B)\leq TC^{S}(f)$.
\end{proposition}

\begin{proof}
Let $TC^{S}(f)=n$. Then there are  open subsets $\big\{ U_i\big\}_{i=0}^{i=n}$ of $X\times X$ and $s_i: U_i \to PY$ such that $s_i(x_0,x_1)(0)= f(x_0)$ and $s_i(x_0,x_1)(1)= f(x_1)$. We define $V_i=U_i \cap {X\times B}$, and  $\sigma_i: V_i\to P_{Y\times f(B)}$ is defined  by $\sigma_i = s_i  \vert_{V_i}$. Therefore $TC^S(f, B) \leq TC^S(f)$.
\end{proof}

In the previous section, we see that targeted complexity of Pave\v{s}ic is homotopy invariant for fibrations if target subspaces are equal. However, the targeted complexity of Scott is homotopy invariant without any other phenomenon.

\begin{proposition}
If $f\simeq g: X\to Y$, then $TC^S(f, B)= TC^S(g, B)$.
\end{proposition}

\begin{proof}
Let $f\simeq g$. Then there is a homotopy $H:X\times \mathbb{I} \to Y$ such that $H(x,0)=f(x)$ and $H(x,1)= g(x)$.
By the symmetry property,  it suffices to prove that $TC^S(g, B) \le TC^S (f, B)$. Hence let $TC^S(f, B)= n$. Then there are open subsets $\lbrace U_i\rbrace_{i=0}^{n}$ of $X\times B$ and $\sigma_i: U_i \to P_{Y\times C}$ such that $\sigma_i(x,b)(0)=f(x)$ and $\sigma_i(x,b)(1)= f(b)$.
Now we set
 $\sigma_i^{\prime}:U_i \to P_{Y\times C}$ by  $\sigma_i^{\prime} (x, b):=
 H^-(x, -)*   \sigma_i (x, b)  * H(b, -)$, where $H^- (x, t) = H(x, 1-t)$. For $t =0$,  it holds that $\sigma_i^{\prime} (x, b)(t)=  H^- (x, 0) = H(x, 1) = g(x)$ and that $\sigma_i^{\prime} (x, b) (1)=  H(b, 1) = g(b)$.
Therefore $TC^S(g, B) \le TC^S(f, B)$. 
\end{proof}

Scott showed that $TC^S  (f) \le TC^P(f)$ for any map. The following proposition shows that it holds for the targeted cases too.

\begin{proposition}
Let $f:X\to Y$ be a map and let $B\subseteq X$. Then $TC^S(f, B)\leq TC^P(f,B)$.
\end{proposition}

\begin{proof}
Let $ TC^P(f,B)= n$. Then for some open subsets $U_0, U_1, \ldots , U_n $ of $X\times f(B)$, there are $s_i: U_i\to P_{X\times B}$ such that $s_i(x, f(b))(0)  = x$ and $s_i(x, f(b)) (1) = b'$ whenever $f(b) = f(b')$. We define $V_i= (1\times f)^{-1}(U_i) \cap X\times B$, and $\sigma_i: V_i\to P_{Y\times C}$ is defined  by $\sigma_i:= f_{\ast} s_i(1\times f)$, where $f_{\ast}: P_{X\times B}\to P_{Y\times C}$ maps each path $\gamma$ to  $f\gamma$. We have $\sigma_i(x,b)(0)= f_{\ast} s_i(1\times f)(x,b)(0)=  f_{\ast} s_i(x, f(b))(0)=  fs_i(x, f(b))(0)= f(x)$ and
 $\sigma_i(x,b)(1)= f_{\ast} s_i(1\times f)(x,b)(1)=  f_{\ast} s_i(x, f(b))(1)=  fs_i(x, f(b))(1)= f(b)$.
 Therefore $TC^S(f, B) \leq TC^P (f, B)$.
\end{proof}

Scott showed that for  any map $f: X \to Y$, the complexity of $f$, $TC^S(f)$ is less than topological complexities of domain and codomain of $f$. It holds for targeted cases too if we replace Farber's topological complexity by relative topological complexity of pair of spaces.

\begin{proposition}
Let $f: X\to Y$ be a map and let  $B\subseteq X$. Then
\[
TC^S(f, B) \leq \min \lbrace TC(X,B), TC(Y, f(B)) \rbrace.
\]
\end{proposition}

\begin{proof}
Let $TC(X,B)= n$. Then for some open subsets $U_0, U_1, \ldots , U_n $ of $X\times B$, there are $s_i: U_i \to P_{X\times B}$ such that $\pi_B s_i= id_{U_i}$.
We define $\sigma_i:U_i\to P_{Y\times C}$ by $\sigma_i(x,b)= f_{\ast}s_i(x,b)$. Recall that $f_{\ast}$ maps each $\gamma$  in $PX$ to path $f\gamma$. Then we have $\sigma_i(x,b)(0)= f_{\ast}s_i(x,b)(0)= f(x)$ and  $\sigma_i(x,b)(1)= f_{\ast}s_i(x,b)(1)= f(b)$. Therefore $TC^S(f, B) \leq TC(X,B)$.

Let $TC(Y,f(B))= n$. Then for some open subsets $U_0, U_1, \ldots , U_n $ of $Y\times f(B)$, there are $s_i: U_i \to P_{Y\times f(B)}$ such that $\pi s_i= id_{U_i}$.
We define $V_i= (f \times f)^{-1}(U_i)\cap X\times B$, and $\sigma_i:U_i\to P_{Y\times f(B)}$ is defined  by  $\sigma_i(x,b)= s_i(f \times f) (x,b)$. Then we have $\sigma_i(x,b)(0)= s_i(f \times f)  (x,b)(0)= s_i(f(x),f(b))= f(x)$ and  $\sigma_i(x,b)(1)= s_if(x,b)(1)=s_i(f(x),f(b))= f(b)$. Therefore $TC^S(f, B) \leq TC(Y,C)$.
\end{proof}

In the following proposition, we investigate the targeted complexity of composition of maps. The original inequality was proved by Scott.

\begin{proposition}
Let $f:(X,B) \to (Y,f(B))$ and $g:(Y,f(B)) \to (Z,gf(B))$ be two maps. Then
\[
TC^S (gf, B)\leq \min \lbrace TC^S(f, B), TC^S(g, f(B)) \rbrace.
\]
\end{proposition}

\begin{proof}
Let $TC^S(f, B)= n$. Then for some open subsets $U_0, U_1, \ldots , U_n $ of $X \times B$, and for each $i$, there is $\sigma_i: U_i \to P_{Y\times f(B)}$ such that $\sigma_i(x,b)(0) = f(x)$ and $\sigma_i(x, b)(1) = f(b)$.  We define $\sigma^{\prime}_i: U_i \to P_{Z\times gf(B)}$ by  $\sigma^{\prime}_i(x,b)= g_{\ast}\sigma_i(x,b)$, where $g_{\ast}: P_{Y\times f(B)}\to P_{Z\times gf(B)}$ maps each $\gamma$ to $g\gamma$. Then we have $\sigma^{\prime}_i(x,b)(0)= g_{\ast}\sigma_i(x,b)(0)= g(\sigma_i(x,b)(0))= gf(x)$ and $\sigma^{\prime}_i(x,b)(1)= g_{\ast}\sigma_i(x,b)(1)= g(\sigma_i(x,b)(1))=gf(b)$.
Therefore $TC^S (gf, B)\leq TC^S(f, B)$.

Now let $TC^S (g, f(B)) = n$. Then for some open subsets $U_0, U_1, \ldots , U_n $ of $Y \times f(B)$, and for $i$, there is $\sigma_i: U_i \to P_{Z\times gf(B)}$ such that $\sigma_i (y, f(b))(0) = g(y)$ and $\sigma_i(y, f(b)) (1) = gf(b)$.  We define $V_i= (f \times f)^{-1}(U_i) \cap X \times B$, and $\sigma^{\prime}_i: V_i \to P_{Z\times gf(B)}$ is defined by  $\sigma^{\prime}_i(x,b)= \sigma_i(f \times f)(x,b)$. Then   $\sigma^{\prime}_i(x,b)(0)= \sigma_i(f(x),f(b))(0)= gf(x)$ and $\sigma_i^{\prime}(x,b)(1)= \sigma_i(f(x),f(b))(1)= gf(b)$.
 Therefore $TC^S (gf, B)\leq TC^S (g, B)$.
\end{proof}

For the composition of maps, Scott's complexity is less than the complexity of all components. The targeted complexity satisfies this property; that is, for each composition of maps, the following inequality holds:
\[
TC^S (f_1 \circ \cdots \circ f_n, B) \le \min \{TC^S (f_1, f_2 \circ \cdots \circ f_n (B)), \ldots, TC^S( f_n, B)\}.
\]

The product of maps has useful roles in robotic motion planning; it establishes the motion of two robots simultaneously, and it is significant to find some bounds for the complexity of maps. The following proposition can be proven by a similar argument as in \cite{bib7}.

\begin{proposition}
Let $f:(X,B) \to (Y, f(B))$ and $g:(Z,D) \to (W, g(D))$ be two maps, and let $Y\times W$ be a normal space. Then $TC^S (f \times g, B \times D) \leq TC^S(f, B) + TC^S (g, D)$.
\end{proposition}


\section{Higher targeted complexity of a map}

In this section, we study the higher complexity of maps with a target. Recall that the higher complexity of a map is a generalization of usual complexity, which increases endpoints to more than two endpoints. In other words, the $n$th higher complexity considers paths as $n$ pieces attached together; see \cite{Bas, Is, bib6}.
In  \cite[Definition 3.7]{Is}, the $n$-dimensional higher topological complexity of $f$  was introduced for fibration maps. Also in  \cite[Proposition 3.8]{Is}, it was proved that $TC_2(f)= TC^{P}(f)$. Now we reformulate $TC_n(f)$ for any map. 

\begin{definition}
Let $f:X\to Y$ be a map and let $p_i:X^n\to X$ be the projection map onto the $i$th component for $1\leq i\leq n$.
Then the $n$-dimensional higher complexity of $f$ is defined as
\[TC_n(f)= D(fp_1, fp_2, \ldots, fp_n),\]
where $D$ denotes the homotopic distance between maps \cite{Mac}. Recall that the homotopic distance between $m$ maps $f_1, \ldots, f_m: X \to Y$ is the least integer $n$ for which there exists an open  covering $\{U_i\}_{i =0}^n$ of $X$ such that $f_j\vert_{U_i} \simeq f_k\vert_{U_i}$ for each $0 \le i \le n$ and $1\le j,k \le m$.
In the above definition, if $f$ is the identity map, then $TC_n(f)= TC_n(X)$. Also, we have $TC_2(f)= TC^{MW}(f)= TC^S(f)$.
Moreover, since the homotopic distance preserves homotopic maps,  immediately, $TC_n(f)$  is homotopy invariant.
\end{definition}

\begin{remark}
If $f\simeq g: X\to Y$, then $TC_n(f)= TC_n(g)$.
\end{remark}

The main goal in this section is to define higher targeted complexity and investigate its properties.
Let $f:X\to Y$ be a map and let $B\subseteq X$. Consider the map $\bar{f}:= f\times f\vert_B: X\times B \to Y\times f(B)$.
\begin{definition}\label{D1}
 Let $f:X\to Y$ be a map and let $\pi_i: X^n\times B^n \to X\times B$ be the projection map onto the $i$th component for $1\leq i\leq n$.
Then the $n$-dimensional targeted complexity of $f$ is
\[
TC_n(f,B)= D(\bar{f}\pi_1, \bar{f}\pi_2, \ldots, \bar{f}\pi_n).
\]
\end{definition}

If $f$ is the identity map, then $TC_n(f, B)= TC_n(X,B)$. Also, we have $TC_2(f,B)= TC^S(f, B)$.
Moreover, if $B=\emptyset $, then $TC_n(f, B)= TC_n(f)$.
Furthermore, if $f\simeq g: X\to Y$, then $TC_n(f, B)= TC_n(g, B)$, because the  homotopic distance $D$, is homotopy invariant.
\begin{proposition}
Let $f:(X,B) \to (Y, f(B))$ and $g:(Y,f(B)) \to (Z,gf(B))$ be two maps. Then
\[
TC_n(gf, B)\leq  TC_n(f, B).
\]
\end{proposition}
\begin{proof}
Since $\overline{gf} = \bar{g} \circ \bar{f}$, by \cite[ Proposition 3.1]{Mac}, we have
\begin{align*}
TC_n(gf, B) &= D(\overline{gf}\pi_1, \overline{gf}\pi_2, \ldots, \overline{gf}\pi_n)\\
&= D(\bar{g}\bar{f}\pi_1, \bar{g}\bar{f}\pi_2, \ldots, \bar{g}\bar{f}\pi_n)\\
& \leq
D(\bar{f}\pi_1, \bar{f}\pi_2, \ldots, \bar{f}\pi_n)= TC_n(f, B).
\end{align*}
\end{proof}

\begin{proposition}
Let $f:(X,B) \to (Y,f(B))$ and $g:(Z,D) \to (W,g(D))$ be two maps and let $Y$ and $W$ be normal spaces. Then
 \[
TC_n(f \times g, B \times D) \leq TC_n(f, B) + TC_n(g, D).
 \]
\end{proposition}
\begin{proof}
Let $h= f\times f\vert_B$ and let $k= g \times g\vert_D$. Then by Definition \ref{D1}, we have
\begin{align*}
TC_n(f\times g, B \times D)& =  D((h\times k) \pi_1, (h\times k)\pi_2, \ldots, (h\times k)\pi_n)\\
&=
D(h\pi_1 \times k\pi_1, h\pi_2\times k\pi_2, \ldots, h\pi_n \times k\pi_n).
\end{align*}
Now it follows from \cite[Theorem 3.13]{bib12} that
\begin{align*}
D(h\pi_1 \times k\pi_1, h\pi_2\times k\pi_2, \ldots, h\pi_n \times k\pi_n) &\leq
D(h\pi_1, h\pi_2, \ldots, h\pi_n)\\
& \quad +
 D(k\pi_1, k\pi_2, \ldots, k\pi_n)\\
 & = TC_n(f, B) + TC_n(g, D).
 \end{align*}
\end{proof}

For the $n$-dimensional complexity of map, Is et al. \cite[Proposition 3.9]{Is} proved that $TC_n(f) \le TC_{n+1} (f)$ . For the targeted case, it is immediately proven by \cite[Proposition 2.5]{bib12}.

\begin{proposition}
Let $f:(X,B) \to (Y,C)$ is a pair map. Then $TC_n(f, B)\leq TC_{n+1}(f, B)$.
\end{proposition}

\section*{Acknowledgments}
This research was supported by a grant from Ferdowsi University of Mashhad--Graduate Studies (No. 3/57099).

%
%
%

\end{document}